\documentclass[a4paper,intlimits,reqno,10pt,oneside]{amsart}
\usepackage[centertags]{amsmath}
\usepackage{amsfonts}
\usepackage{amssymb}
\usepackage{amsthm}
\usepackage{graphicx}

\newcommand{\Natural}{\mathbb N}

\newcommand{\Real}{\mathbb R}
\newcommand{\ereal}{(-\infty,+\infty]}

\newcommand{\Rnneg}{[0,+\infty)}
\newcommand{\abs}[1]{\left\vert#1\right\vert}
\newcommand{\set}[1]{\left\{#1\right\}}
\newcommand{\cardinality}[1]{\abs{#1}}
\newcommand{\too}[1]{\stackrel{#1}{\to}}
\newcommand{\nullset}{{\mathcal N}}
\newcommand{\domain}{{\mathcal D}}
\newcommand{\kindex}{r}

\newcommand{\dens}{{\rm dens}}

\newcommand{\norm}[1]{\left\Vert#1\right\Vert}
\newcommand{\newnorm}[1]{\left\Vert|#1|\right\Vert}
\newcommand{\finitetop}[1]{\left\lceil#1\right\rceil}

\newcommand{\closedball}[1]{B_{#1}}
\newcommand{\openball}[1]{B^O_{#1}}

\newcommand{\Banach}{X}

\newcommand{\PRI}{\set{P_\alpha}_{\omega\leq\alpha\leq \mu}}
\newcommand{\linf}{\ell^\infty}
\newcommand{\smooth}{$C^\infty$-smooth}
\newcommand{\ksmooth}[1]{$C^{#1}$-smooth}
\newcommand{\COne}{$C^1$-smooth}
\newcommand{\ointerval}{\Lambda}

\theoremstyle{plain}
\newtheorem{thm}{Theorem}[section]
\newtheorem{cor}[thm]{Corollary}
\newtheorem{lem}[thm]{Lemma}

\theoremstyle{definition}
\newtheorem{defn}[thm]{Definition}

\begin{document}
\title{$C^k$-smooth approximations of LUR norms}
\author{Petr H\'ajek}
\address{Mathematical Institute\\Czech Academy of Science\\\v Zitn\'a 25\\115 67 Praha 1\\Czech Republic}
\email{hajek@math.cas.cz}
\author{Anton\'\i n Proch\'azka}
\address{Charles University\\Sokolovsk\'a 83\\186 75 Praha 8\\Czech Republic
and  Universit\'e Bordeaux 1, 351 cours de
la liberation, 33405, Talence, France.} \email{protony@karlin.mff.cuni.cz}

\date{January 2009}
\thanks{Supported by grants: Institutional Research Plan AV0Z10190503, A100190502, GA \v CR 201/07/0394}

\begin{abstract}
Let $X$ be a WCG Banach space admitting a $C^k$-Fr\' echet smooth norm. Then $X$ admits an equivalent norm which is simultaneously
$C^1$-Fr\' echet smooth, LUR, and a uniform limit of $C^k$-Fr\' echet smooth norms. If $X=C([0,\alpha])$, where $\alpha$ is an ordinal, then
the same conclusion holds true with $k=\infty$.
\end{abstract}

\maketitle

\section{Introduction}
The theory of $C^k$-Fr\' echet smooth approximations of continuous functions on Banach spaces is well-developed,
thanks to the work of many mathematicians, whose classical results and references can be found
in the authoritative monograph \cite{DGZ}. The known techniques
rely on the use of $C^k$-Fr\' echet smooth partititons of unity, resp. certain coordinatewise smooth embeddings into
the space $c_0(\Gamma)$ (due to Torunczyk \cite{To}). They are highly nonlinear, and even non-Lipschitz in nature.
For example, if the given function is Lipschitz or has some uniform continuity, trying to preserve the lipschitzness of the approximating smooth functions
leads to considerable additional technical difficulties (e.g. \cite{F1}, \cite{F2}, \cite{HJ1}, \cite{HJ2}, \cite{J}, \cite{JTZ}).

It is well-known that the (apparently harder, and less developed) parallel theory of approximations of norms on a
Banach space by $C^k$-Fr\' echet smooth renormings requires different techniques.

Several open problems proposed in \cite{DGZ} are addressing these issues.  In particular, if a Banach space
admits an equivalent $C^k$-Fr\' echet smooth renorming,
 is it possible to approximate (uniformly on bounded sets) all norms by $C^k$-Fr\' echet smooth
norms? Even in the separable case, the answer is not known in full generality, although the positive results in
\cite{DFH1} and \cite{DFH2} are quite strong, and apply to most classical Banach spaces.
In the nonseparable setting, no general
results are known, with a small exception of \cite{FHZ}. In particular, one of the open problems in \cite{DGZ}
is whether on a given  WCG Banach space with an equivalent $C^k$-Fr\' echet smooth norm, there exists an equivalent
locally uniformly rotund (LUR) norm which is a uniform
limit on bounded sets of $C^k$-Fr\' echet smooth norms.
The notion of LUR is of fundamental importance for renorming theory, and we refer to \cite{DGZ} and the more
recent \cite{MOTV} for an extensive list of authors and results.

Such a result is of interest for several reasons. It can be used to
obtain rather directly the uniform approximations of general continuous operators, by $C^k$-Fr\' echet smooth ones.
Moreover, since LUR norms form a residual set in the metric space of all equivalent norms on a Banach space,
a positive answer is to be expected. There is a closely related problem of obtaining a norm which shares simultaneously
good rotundity and smoothness properties. By a famous result of Asplund \cite{A},
on every separable Asplund space there exists
an equivalent norm which is simultaneously $C^1$-Fr\' echet smooth and LUR. A clever proof using Baire category,
and disposing of the separability condition on the underlying Banach space,
was devised in \cite{FZZ} (\cite{DGZ}, II.4.3). The theorem holds in particular in all WCG Asplund spaces (in particular all
reflexive spaces). Its proof works under the assumption that the space admits an LUR norm, as well as a norm whose dual
is LUR. It is well-known that dual LUR implies that the original norm is $C^1$-Fr\' echet smooth.
However, using this approach one cannot in general handle norms with higher degree of differentiability, even in the separable case.
 Indeed, by \cite{FWZ}
(\cite{DGZ}, Proposititon V.1.3), a space admitting a LUR and simultaneously $C^2$-Fr\' echet smooth norm is
superreflexive. There is not even a rotund and $C^2$-Fr\' echet smooth norm on $c_0(\Gamma)$ (\cite{H1}, \cite{H2}).
In fact, one cannot even handle the proper case of LUR and $C^1$-Fr\' echet smooth norms. Indeed, Talagrand
\cite{T} proved that $C([0,\omega_1])$ admits an equivalent $C^\infty$-Fr\' echet smooth norm, although it admits no
dual LUR renorming. The existence of LUR renorming of this space follows from Troyanski's theorem \cite{Tr}.
In light of the previous results it is natural to ask whether this space has a $C^1$-Fr\' echet smooth and
simultaneously LUR renorming. This question was posed on various occasions, e.g. in \cite{FMZ}.

Our main result addresses both of the above mentioned open problems, namely higher smoothness approximation and simultaneous
LUR and \COne ness. Under reasonable assumptions (e.g. for WLD, $C(K)$ where $K$ is Valdivia compact,
or $C([0,\alpha])$, i.e. the space of continuous functions on the ordinal interval $[0,\alpha]$), it gives a renorming which is
simultaneously $C^1$-Fr\' echet smooth and LUR, and admits a uniform approximation on bounded sets by
$C^k$-Fr\' echet norms.
As a corollary we obtain a positive solution to both of the mentioned problems.
We should emphasize that it is unknown whether $C^1$-Fr\' echet smooth norms are residual, or even dense, in the space
of all equivalent norms on $C([0,\alpha])$.

The paper is organized as follows. In Section~\ref{s:Preliminaries}, we introduce our notation and we present some auxiliary
lemmata. We include the easy proofs for reader's convenience. The main result, its corollaries and the frame of the proof of
 the main result are gathered in Section~\ref{s:MR}. Sections~\ref{s:AboutN} and~\ref{s:AboutJ} then contain the
details of the construction.

\section{Preliminaries}\label{s:Preliminaries}
The closed unit ball of a Banach space $(X,\norm{\cdot})$ is denoted by $\closedball{(X,\norm{\cdot})}$, or $\closedball{X}$
for short. Similarly, the open unit ball of $X$ is $\openball{(X,\norm{\cdot})}=\openball{X}$. By $\Gamma$
we denote an index set. Smoothness and higher smoothness is meant in the Fr\' echet sense.
\begin{defn}\label{d:LDF}
Let $A \subset \linf(\Gamma)$.
We say that a function $f:\linf(\Gamma) \to \Real$ in $A$ \emph{locally depends on finitely many coordinates} (LFC) if for each $x \in A$ there exists a neighborhood $U$ of $x$, a finite $M=\set{\gamma_1,\ldots,\gamma_n} \subset \Gamma$ and a function $g:\Real^{\abs{M}} \to \Real$ such that $f(y)=g(y(\gamma_1),\ldots,y(\gamma_n))$ for each $y \in U$.
\end{defn}

\begin{defn}
Let $X$ be a vector space. A function $g:X \to \linf(\Gamma)$ is said to be \emph{coordinatewise convex} if, for each $\gamma \in \Gamma$, the function $x \mapsto g_\gamma(x)$ is convex.
We use the terms as \emph{coordinatewise non-negative} or \emph{coordinatewise \ksmooth{k}} in a similar way.
\end{defn}

\begin{lem}\label{l:composedLFC}
Let $\Banach$ be a Banach space and let $h:\Banach \to \linf(\Gamma)$ be a continuous function which is coordinatewise \ksmooth{k}, $k \in \Natural \cup \set{\infty}$. Let $f:\linf(\Gamma) \to \Real$ be a \ksmooth{k} function which locally depends on finitely many coordinates. Then $f\circ h$ is \ksmooth{k}.
\end{lem}
\begin{proof}
Let $x\in X$ be fixed. Since $f$ is LFC, there is a neighborhood $U$ of $h(x)$, $M=\set{\gamma_1,\ldots,\gamma_n}\subset \Gamma$ and $g:\Real^{\abs{M}} \to \Real$ as in Definition~\ref{d:LDF}. The function $g$ is \ksmooth{k}, because $f$ is \ksmooth{k}. As $h$ is continuous, there exists a neighborhood $V$ of $x$ such that $h(V) \subset U$. Since $h$ is coordinatewise \ksmooth{k}, it follows that $h(\cdot)\upharpoonright_{ M}:=(h(\cdot)(\gamma_1),\ldots,h(\cdot)(\gamma_n))$ is \ksmooth{k} from $X$ to $\Real^{\abs{M}}$. Finally, we have for each $y \in V$ that
$f( h(y))=g (h(y)\upharpoonright_{M})$ and the claim follows.
\end{proof}

\begin{lem}\label{l:implicitLFC}
Let $\Phi:\linf(\Gamma) \to \Real$ and let $x \in \linf(\Gamma)$ be such that
\begin{list}{}{}
\item{a)} $\Phi$ is LFC at $x$,
\item{b)} $\Phi'(x)x \neq 0$,
\item{c)} $\Phi(\cdot)$ and $\Phi'(\cdot)$ are continuous at $x$.
\end{list}
Then there is a neighborhood $U$ of $x$ and a unique function $F:U \to \Real$ which is continuous at $x$ and satisfies $F(x)=1$ and $\Phi(\frac{y}{F(y)})=1$ for all $y \in U$. Moreover $F$ is LFC at $x$.
\end{lem}
\begin{proof}
The first part of the assertion follows immediately from the Implicit Function Theorem. We will show that $F$ is LFC at $x$. From the assumption a) we know that there is a neighborhood $V$ of $x$, $M=\set{\gamma_1,\ldots,\gamma_n} \subset \Gamma$, and $g:\Real^n \to \Real$ such that $\Phi(y)=g(y\upharpoonright_{M})$ for all $y \in V$. It is obvious that $g'(x\upharpoonright_{M})x\upharpoonright_{M}=\Phi'(x)x$ so it is possible to apply the Implicit Function Theorem to the equation $g\left(\frac{y}{h(y)}\right)=1$ to get $h:V' \to \Real$, where $V'$ is a neighborhood of $x\upharpoonright_{M}$, such that $h(x\upharpoonright_{M})=1$ and  $h$ is continuous at $x\upharpoonright_{M}$.
There is a neighborhood $U'\subset U \cap V$ of $x$ such that we may define $H:U' \to \Real$ by $H(y):=h(y\upharpoonright_{M})$ for $y \in U'$. Then $H(x)=1$ and $H$ is continuous at $x$. Also, $\Phi\left(\frac{y}{H(y)}\right)=g\left(\frac{y\upharpoonright_{M}}{h(y\upharpoonright_{M})}\right)=1$. The uniqueness of $F$ implies that $F=H$ on $U'$ so $F$ is LFC at $x$.
\end{proof}

The following lemma is a variant of Fact II.2.3(i) in~\cite{DGZ}.
\begin{lem}\label{l:equivLUR}
Let $\varphi:\Banach \to \Real$ be a convex non-negative function, $x_\kindex ,x \in \Banach$ for $\kindex \in \Natural$. Then the following conditions are equivalent:
\begin{list}{}{}
\item{\rm (i)} $\frac{\varphi^2(x_\kindex)+\varphi^2(x)}{2}-\varphi^2(\frac{x+x_\kindex}{2}) \to 0$,
\item{\rm (ii)} $\lim \varphi(x_\kindex)=\lim \varphi(\frac{x+x_\kindex}{2})=\varphi(x)$.
\end{list}
If $\varphi$ is homogeneous, the above conditions are also equivalent to
\begin{list}{}{}
\item{\rm (iii)} $2\varphi^2(x_\kindex)+2\varphi^2(x)-\varphi^2(x+x_\kindex) \to 0$.
\end{list}
\end{lem}
\begin{proof}
Since $\varphi$ is convex and non-negative, and $y \mapsto y^2$ is increasing for $y \in [0,+\infty)$, it holds
\[
\frac{\varphi^2(x_\kindex)+\varphi^2(x)}{2}-\varphi^2\left(\frac{x+x_\kindex}{2}\right) \geq
\frac{\varphi^2(x_\kindex)+\varphi^2(x)}{2}-\left(\frac{\varphi(x)+\varphi(x_\kindex)}{2}\right)^2= \left(\frac{\varphi(x)-\varphi(x_\kindex)}{2}\right)^2
\]
which proves (i) $\Rightarrow$ (ii). The implication (ii) $\Rightarrow$ (i) is trivial and so is the equivalence (i) $\Leftrightarrow$ (iii).
\end{proof}

\begin{lem}\label{l:sufficientConvexity}
Let $f, g$ be twice differentiable, convex, non-negative, real functions of one real variable.
Let $F:\Real^2 \to \Real$ be given as $F(x,y):=f(x)g(y)$. For $F$ to be convex in $\Real^2$, it is sufficient that $g$ is convex and
\begin{equation}\label{e:convexityCondition}
(f'(x))^2(g'(y))^2\leq f''(x)f(x)g''(y)g(y).
\end{equation}
for all $(x,y) \in \Real^2$.
\end{lem}
\begin{proof}
Let $(x,y) \in \Real^2$ be fixed. Since $g$ is convex, the function $F$ is convex when restricted to the vertical line
going through $(x,y)$. Let $s=at+b$ ($a,b \in \Real$) be a line going through $(x,y)$, i.e. $y=ax+b$.
The second derivative at a point $(x,y)$ of $F$ restricted to this line is given as:
\[
f(x)g''(y)a^2+2f'(x)g'(y)a+f''(x)g(y).
\]
In order for the second derivative to be non-negative for all $a \in \Real$, it is sufficient that the discriminant
$(2f'(x)g'(y))^2-4f(x)g''(y)f''(x)g(y)$ of the above quadratic term be non-positive, which occurs exactly when our condition \eqref{e:convexityCondition} holds for $(x,y)$.
\end{proof}

\begin{defn}
We say that a function $f: \linf(\Gamma) \to \Real$ is \emph{strongly lattice} if $f(x)\leq f(y)$ whenever $\abs{x(\gamma)} \leq \abs{y(\gamma)}$ for all $\gamma \in \Gamma$.
\end{defn}

\begin{lem}\label{l:SLconvexComposition}
Let $f:\linf(\Gamma) \to \Real$ be convex and strongly lattice. Let $g:X \to \linf(\Gamma)$ be coordinatewise convex and coordinatewise non-negative. Then $f \circ g:X \to \Real$ is convex.
\end{lem}
\begin{proof}
Let $a,b \geq 0$ and $a+b=1$.
Since $g$ is coordinatewise convex and non-negative, we have
\[
0\leq g_\gamma(ax+by)\leq a g_\gamma(x)+ b g_\gamma(y)
\]
for each $\gamma \in \Gamma$.
The strongly lattice property and the convexity of $f$ yield
\[
f(g(ax+by)) \leq f(a g(x)+ b g(y)) \leq a f(g(x))+ b f(g(y))
\]
so $f \circ g$ is convex.
\end{proof}

\begin{defn}
Let us define $\finitetop{\cdot}:\linf(\Gamma) \to \Real$ by $\finitetop{x}=\inf\set{t; \set{\gamma;\abs{x(\gamma)}>t} \mbox{ is finite}}$.
Then $\finitetop{\cdot}$ is 1-Lipschitz, strongly lattice seminorm on $(\linf(\Gamma),\norm{\cdot}_\infty)$.
\end{defn}
\begin{proof}
In fact $\finitetop{x}=\norm{q(x)}_{\linf/c_0}$, where $q: \linf(\Gamma) \to  \linf(\Gamma)/c_0(\Gamma)$ is the quotient map  and $\norm{\cdot}_{\linf/c_0}$ the canonical norm on the quotient $\linf(\Gamma)/c_0(\Gamma)$. Clearly, $\finitetop{x}=0$ if and only if $x \in c_0(\Gamma)$. Let us assume that $\finitetop{x}=t>0$. Then, for every $0<s<t$, there are infinitely many $\gamma \in \Gamma$ such that $\abs{x(\gamma)}>s$. It follows that $\norm{x-y}_\infty>s$ for every $y \in c_0(\Gamma)$ and consequently $\norm{q(x)}_{\linf/c_0} \geq t$. On the other hand, we may define $y \in c_0(\Gamma)$ as
\[
y(\gamma):=\begin{cases}
x(\gamma)-t \mbox{ if } x(\gamma)>t,\\
x(\gamma)+t \mbox{ if } x(\gamma)<-t,\\
0 \mbox{ otherwise.}
\end{cases}
\]
Obviously $\norm{x-y}_\infty \leq t$, so $\norm{q(x)}_{\linf/c_0} \leq t$. The strongly lattice property of $\finitetop{\cdot}$ follows directly from the definition.
\end{proof}

\begin{defn} Let $(X,\norm{\cdot})$ be a Banach space and let $\mu$ be the smallest ordinal such that $\cardinality{\mu}=\dens(X)$.
A system $\PRI$ of projections from $X$ into $X$ is called a \emph{projectional resolution of identity} (PRI) provided that, for every $\alpha \in [\omega,\mu]$, the following conditions hold true
\item{(a)} $\norm{P_\alpha}=1$,
\item{(b)} $P_\alpha P_\beta= P_\beta P_\alpha = P_\alpha$ for $\omega \leq \alpha \leq \beta \leq \mu$,
\item{(c)} $\dens(P_\alpha X) \leq \cardinality{\alpha}$,
\item{(d)} $\bigcup\set{P_{\beta+1} X: \beta < \alpha}$ is norm-dense in $P_\alpha X$,
\item{(e)} $P_\mu = id_X$.
\end{defn}

If $\PRI$ is a PRI on a Banach space $X$, we use the following notation:
 $\ointerval:=\set{0} \cup [\omega,\mu)$, $Q_\gamma:=P_{\gamma+1}-P_\gamma$ for all $\gamma \in [\omega,\mu)$ while $Q_0:=P_\omega$, and $P_A:=\sum_{\gamma \in A} Q_\gamma$ for any finite subset $A$ of $\ointerval$.
\begin{lem}\label{l:PRI}
Let $\Banach$ be a Banach space with a PRI $\PRI$. Then for each $x \in X$, $\varepsilon>0$, $\alpha \in [\omega,\mu]$ there is a finite set $A_\varepsilon^\alpha(x)\subset \ointerval$ such that
\[
\norm{P_{A_\varepsilon^\alpha(x)}x-P_\alpha x}<\varepsilon.
\]
We may choose $A=A_\varepsilon^\alpha(x)$ in such a way that $Q_\beta x \neq 0$ for $\beta \in A$ since $P_A=\sum_{\gamma \in A} Q_\gamma$.
\end{lem}

\begin{proof}
We will proceed by a transfinite induction on $\alpha$. If $\alpha=\omega$, then $A^\omega_\varepsilon(x):=\set{0}$ for any $\varepsilon>0$. If $\alpha=\beta+1$ for some ordinal $\beta$, then $A^\alpha_\varepsilon(x):=A^\beta_\varepsilon(x) \cup \set{\beta}$ for all $\varepsilon>0$. Finally, if $\alpha$ is a limit ordinal, we will use the continuity of the mapping $\gamma \mapsto P_\gamma x$ at $\alpha$ \cite[Lemma VI.1.2]{DGZ} to find $\beta<\alpha$ such that
$\norm{P_\beta x- P_\alpha x}<\varepsilon/2$. Thus it is possible to set $A^\alpha_\varepsilon(x):=A^\beta_{\varepsilon/2}(x)$.
\end{proof}

\section{Main Result}\label{s:MR}
Let us recall that a norm $\norm{\cdot}$ in a Banach space $X$ is \emph{locally uniformly rotund} (LUR) if  $\lim_\kindex \norm{x_\kindex-x}=0$ whenever $\lim_\kindex\left(2\norm{x_\kindex}^2+2\norm{x}^2-\norm{x_\kindex + x}^2\right)=0$. 
\begin{thm}\label{t:main}
Let $k\in \Natural \cup \set{\infty}$. Let $(X,\abs{\cdot})$ be a Banach space with a PRI $\PRI$ such that
each $Q_\gamma X$ admits a \COne, LUR equivalent norm which is a limit (uniform on bounded sets) of \ksmooth{k} norms. Let $X$ admit an equivalent \ksmooth{k} norm $\norm{\cdot}$.

Then $X$ admits an equivalent \COne, LUR norm $\newnorm{\cdot}$ which is a limit (uniform on bounded sets) of \ksmooth{k} norms.
\end{thm}
Our first corollary provides a positive solution of Problem 8.2 (c) in~\cite{FMZ}.
\begin{cor}\label{c:Calpha}
Let $\alpha$ be an ordinal. Then the space $C([0,\alpha])$ admits an equivalent norm which is \COne, LUR and a limit of \smooth\ norms.
\end{cor}
\begin{proof}[Proof of Corollary~\ref{c:Calpha}]
By a result of Talagrand \cite{T} and Haydon \cite{Hay}, $C([0,\alpha])$ admits an equivalent \smooth\ norm. On the other hand, the natural PRI on $C([0,\alpha])$ defined as
\[
(P_\gamma x)(\beta)=
\begin{cases}
x(\beta) \mbox{ if } \beta \leq \gamma,\\
x(\gamma) \mbox{ if } \beta \geq \gamma
\end{cases}
\]
has $(P_{\gamma+1}-P_\gamma)X$ one-dimensional for each $\gamma \in [\omega_0,\alpha)$.
\end{proof}

\begin{thm}
Let $k\in \Natural \cup \set{\infty}$. Let $\mathcal P$ be a class of Banach spaces such that every $X$ in  $\mathcal P$
\begin{itemize}
\item admits a PRI $\PRI$ such that $(P_{\alpha+1}-P_\alpha)X \in \mathcal P$,
\item admits a \ksmooth{k} equivalent norm.
\end{itemize}

Then each $X$ in $\mathcal P$ admits an
equivalent, LUR, \COne\ norm  which is a limit (uniform on bounded sets) of \ksmooth{k} norms.
\end{thm}

\begin{proof}
We will carry out induction on the density of $X$. Let $X \in \mathcal P$ be separable, i.e. $\dens(X)=\omega$. Then we get the result from the theorem of McLaughlin, Poliquin, Vanderwerff and Zizler~\cite{MPVZ} or \cite[Theorem V.1.7]{DGZ}.

Next, we assume for $X \in \mathcal P$ that $\dens(X)=\mu$ and that every Banach space $Y\in \mathcal P$ with $\dens(Y)<\mu$ admits a \COne, LUR norm which is a limit of \ksmooth{k} norms. Let $\PRI$ be a PRI on $X$ such that $Q_\alpha X\in \mathcal P$ for each $\alpha \in \ointerval$. Then $\dens(Q_\alpha X) \leq \cardinality{\alpha+1}=\cardinality{\alpha} < \mu$. Thus the inductive hypothesis enables us to use Theorem~\ref{t:main}.
\end{proof}

The above theorem has immediate corollaries for each $\mathcal P$-class (see \cite{HMVZ} for this notion).
The following Corollary~\ref{c:Vasak} solves in the affirmative Problem~8.8~(s) in~\cite{FMZ} (see also Problem VIII.4 in~\cite{DGZ}).
\begin{cor}\label{c:Vasak} 
Let $X$ admit a \ksmooth{k} norm for some $k \in \Natural \cup {\infty}$. If $X$ is Va\v s\'ak (i.e. WCD) or WLD or $C(K)$ where $K$ is a Valdivia compact, then $X$ admits a \COne, LUR equivalent norm which is a limit (uniform on bounded sets) of \ksmooth{k} norms.
\end{cor}

\begin{proof}[Proof of Theorem~\ref{t:main}]
Let $0<c<1$. It follows from the hypothesis that, for each $\gamma \in \ointerval$, there are a \COne, LUR norm $\norm{\cdot}_\gamma$ on $Q_\gamma X$ and \ksmooth{k} norms $(\norm{\cdot}_{\gamma,i})_{i\in \Natural}$ on $Q_\gamma X$ such that
\begin{equation}\label{e:inductiveHypo}
c\norm{x} \leq \norm{x}_\gamma \leq \norm{x}
\end{equation}
for all $x \in Q_\gamma X$ and such that  $(1-\frac{1}{i^2})\norm{x}_\gamma\leq\norm{x}_{\gamma,i}\leq \norm{x}_\gamma$ for all $x \in Q_\gamma X$.

We seek the new norm on $X$ in the form
\[
\newnorm{x}^2:=N(x)^2+J(x)^2+\norm{x}^2.
\]
We will insure during the construction that both $N$ and $J$ are \COne\ and approximated by \ksmooth{k}\ norms.
In order to see that $\newnorm{\cdot}$ is LUR, we are going to show that $\norm{x-x_\kindex} \to 0$ provided that
\begin{equation}\label{e:LURmain}
2\newnorm{x_\kindex}^2+2\newnorm{x}^2 - \newnorm{x+x_\kindex}^2 \to 0 \mbox{ as } \kindex \to \infty.
\end{equation}
Consider the following two statements:\\
a) $\norm{P_A x_\kindex - P_A x} \to 0$ for each finite $A \subset \ointerval$ with $0\notin \set{Q_\gamma x :\gamma \in A}$,\\
b) for every $\varepsilon>0$ there exists a finite $A \subset \ointerval$ with $0\notin \set{Q_\gamma x :\gamma \in A}$ and such that $\norm{P_A x - x}<\varepsilon$ and $\norm{P_A x_\kindex -x_\kindex}<\varepsilon$ for all but finitely many $\kindex \in \Natural$.

Clearly, the simultaneous validity of a) and b) implies that $\norm{x-x_\kindex} \to 0$ as
\[
\norm{x-x_\kindex} \leq \norm{P_A x_\kindex - P_A x}+ \norm{P_A x - x} + \norm{P_A x_\kindex -x_\kindex}.
\]
We construct $N$ in such a way that we can prove in Lemma~\ref{l:blockconvergence} that \eqref{e:LURmain} implies a). Consequently, we construct $J$ in such a way that we can prove in Lemma~\ref{l:finalHit} that \eqref{e:LURmain} implies b).
\end{proof}
\section{About $N$}\label{s:AboutN}
We may and do assume that the equivalent norms $\abs{\cdot}$ and $\norm{\cdot}$ satisfy
\[
\abs{\cdot} \leq \norm{\cdot} \leq C \abs{\cdot}
\]
for some $C\geq 1$.

The basic properties of PRI~\cite[Lemma VI.1.2]{DGZ} and the above equivalence of norms yield $(\norm{Q_\gamma x})_{\gamma \in \ointerval} \in c_0(\ointerval)$, and using $\norm{Q_\gamma}\leq 2C$ with the second inequality of \eqref{e:inductiveHypo}, it follows that $T:x \in (X,\norm{\cdot}) \mapsto (\norm{Q_\gamma x}_\gamma)_{\gamma \in \ointerval} \in (c_0(\ointerval),\norm{\cdot}_\infty)$ is a $2C$-Lipschitz mapping. Similarly for $T_i:x \in X \mapsto (\norm{Q_\gamma x}_{\gamma,i})_{\gamma \in \ointerval} \in c_0(\ointerval)$.

For each $n \in \Natural$, we will consider an equivalent norm on $c_0(\ointerval)$ given as
\[
\zeta_n(x):=\sup_{M\in\ointerval_n} \sqrt{\sum_{\gamma \in M} x(\gamma)^2}
\]
where $\ointerval_n:=\set{M \in 2^\ointerval: \cardinality{M}=n}$. It is easily seen that $\zeta_n$ is $n$-Lipschitz with
respect to the usual norm on $c_0(\Lambda)$. Also, $\zeta_n$ is obviously strongly lattice, so by Theorem~1 in~\cite{FHZ}, for each $\varepsilon>0$ there is a \smooth\ equivalent norm $N_{n,\varepsilon}$ on $c_0(\ointerval)$ such that
$(1-\varepsilon)\zeta_n(x)\leq N_{n,\varepsilon}(x)\leq \zeta_n(x)$
for all $x \in c_0(\ointerval)$ with $\norm{x}_\infty \leq 1$.
Finally, we define
\[
N(x)^2:=\sum_{m,n \in \Natural} \frac{1}{2^{n+m}} N^2_{n,\frac{1}{m}}(T(x))
\]
Now the norm $N(\cdot)$ is \COne\ since each $N_{n,\frac{1}{m}}\circ T$ is $2nC$-Lipschitz and \COne. The latter property follows since each $N_{n,\frac{1}{m}}$ is not only LFC but it depends  on nonzero coordinates only (cf. Remark on page 461 in~\cite{Hay}).
This fact is not explicitely mentioned in~\cite{FHZ} but follows from the proof there (see~\cite[p. 270]{FHZ}).
We may define the approximating norms as
\[
N_i(x)^2:=\sum_{m,n=1}^i \frac{1}{2^{n+m}} N^2_{n,\frac{1}{m}}(T_i(x)).
\]
As a finite sum of \ksmooth{k} norms, $N_i$ is \ksmooth{k}. Using $\norm{T_i(x)-T(x)}_\infty \leq \frac{2C}{i^2}$ for $\norm{x}\leq 1$, it is standard to check that $N_i(x) \to N(x)$ uniformly for $\norm{x}\leq 1$. We carry out some similar considerations in more detail on page~\pageref{e:similarconsiderations} when we demonstrate that $J$ is approximated by \ksmooth{k} norms.
\begin{lem}\label{l:blockconvergence}
Let us assume that \eqref{e:LURmain} holds for $x, x_\kindex \in X$, $\kindex \in \Natural$, and let $\tilde{A} \subset \ointerval$ be a finite set such that $Q_\gamma x \neq 0$ for $\gamma \in \tilde{A}$. Then $\norm{P_{\tilde{A}} x - P_{\tilde{A}} x_\kindex} \to 0$ as $k \to \infty$.
\end{lem}
\begin{proof}
Let $A:=\set{\gamma \in \ointerval: \norm{Q_\gamma x}_\gamma \geq \min_{\alpha \in \tilde{A}} \norm{Q_\alpha x}_\alpha}$. Let $n:=\cardinality{A}$. We may assume that $\newnorm{x}\leq 1$ which implies $\norm{Tx}_\infty \leq 2C$. Using \eqref{e:LURmain} and Lemma~\ref{l:equivLUR} we may assume that $\newnorm{x_\kindex} \leq 2$ thus $\norm{Tx}_\infty \leq 4C$.
The convergence \eqref{e:LURmain} and convexity (see Fact II.2.3 in~\cite{DGZ}) imply that
\[
2N_{n,\frac{1}{m}}^2(T(x_\kindex)) +2N_{n,\frac{1}{m}}^2(T(x))-N_{n,\frac{1}{m}}^2(T(x+x_\kindex)) \too{\kindex} 0
\]
for all $m \in \Natural$.
This further yields that
\[
2\zeta_n^2(T(x_\kindex))+2\zeta_n^2(T(x))-\zeta_n^2(T(x+x_\kindex)) \too{\kindex} 0
\]
as well.
Indeed, let $\varepsilon>0$ be given. We use that $N_{n,\frac{1}{m}} \to \zeta_n$ uniformly on bounded sets of $c_0(\ointerval)$ to find $m_0 \in \Natural$ such that
$\abs{N_{n,\frac{1}{m}}^2(y)- \zeta_n^2(y)}<\varepsilon/6$ for all $y \in 6C\closedball{c_0(\ointerval)}$ and all $m\geq m_0$. Now let $\kindex_0\in \Natural$ satisfy that for all $\kindex \geq \kindex_0$ it holds
$2N_{n,\frac{1}{m_0}}^2(T(x_\kindex)) +2N_{n,\frac{1}{m_0}}^2(T(x))-N_{n,\frac{1}{m_0}}^2(T(x+x_\kindex))<\varepsilon/6$. For each $\kindex \geq \kindex_0$ we obtain
$2\zeta_n^2(T(x_\kindex))+2\zeta_n^2(T(x))-\zeta_n^2(T(x+x_\kindex))<\varepsilon$.

Let $B\in \ointerval_n$ be arbitrary and let $A_\kindex \in \ointerval_n$ such that
\[
\sqrt{\sum_{\gamma \in A_\kindex} \norm{Q_\gamma (x+x_\kindex)}_\gamma^2}=\zeta_n(x+x_\kindex).
\]
Then
\begin{equation}\label{e:Deville}
\begin{split}
2\zeta_n^2(T(x_\kindex))+2\zeta_n^2(T(x))-\zeta_n^2(T(x+x_\kindex))
&\geq 2\sum_{\gamma \in B} \norm{Q_\gamma x}_\gamma^2+2\sum_{\gamma \in A_\kindex} \norm{Q_\gamma x_\kindex}_\gamma^2-\sum_{\gamma \in A_\kindex} \norm{Q_\gamma (x+x_\kindex)}_\gamma^2\\
&= 2\sum_{\gamma \in A_\kindex} \norm{Q_\gamma x}_\gamma^2+2\sum_{\gamma \in A_\kindex} \norm{Q_\gamma x_\kindex}_\gamma^2-\sum_{\gamma \in A_\kindex} \norm{Q_\gamma (x+x_\kindex)}_\gamma^2\\ &+ 2\left(\sum_{\gamma \in B} \norm{Q_\gamma x}_\gamma^2 - \sum_{\gamma \in A_\kindex} \norm{Q_\gamma x}_\gamma^2\right)
\end{split}
\end{equation}
Since \[
2\sum_{\gamma \in A_\kindex} \norm{Q_\gamma x}_\gamma^2+2\sum_{\gamma \in A_\kindex} \norm{Q_\gamma x_\kindex}_\gamma^2-\sum_{\gamma \in A_\kindex} \norm{Q_\gamma (x+x_\kindex)}_\gamma^2\geq 0
\]
we get from \eqref{e:Deville} that
\begin{equation}\label{e:liminf}
\liminf_\kindex \sum_{\gamma \in A_\kindex} \norm{Q_\gamma x}_\gamma^2 \geq \sup\set{ \sum_{\gamma \in B} \norm{Q_\gamma x}_\gamma^2:B \in \ointerval_n}=\zeta_n(T(x))=\sum_{\gamma \in A} \norm{Q_\gamma x}_\gamma^2
\end{equation}
where the last equality follows from the definition of $A$. Equation \eqref{e:liminf} together with the definition of $A$ show that $A=A_\kindex$ for all $\kindex$ sufficiently large. We continue with such $\kindex$ and we choose $B:=A$ in \eqref{e:Deville} to get that
\[
2\sum_{\gamma \in A} \norm{Q_\gamma x}_\gamma^2+2\sum_{\gamma \in A} \norm{Q_\gamma x_\kindex}_\gamma^2-\sum_{\gamma \in A} \norm{Q_\gamma (x+x_\kindex)}_\gamma^2 \too{\kindex} 0.
\]
Since $x \mapsto \sqrt{\sum_{\gamma \in A} \norm{Q_\gamma x}_\gamma^2}$ is an equivalent LUR norm on $P_A X$, it follows that $\norm{P_A (x-x_\kindex)} \to 0$ and, by continuity of $P_{\tilde{A}}$, we obtain the claim of the lemma.
\end{proof}

\section{About $J$}\label{s:AboutJ}
Let $\set{\phi_\eta}_{0<\eta<1}$ be a system of functions satisfying
\begin{list}{}{}
\item{(i)} $\phi_\eta:\Rnneg \to \Rnneg$, for $0<\eta<1$, is a convex \smooth\ function such that $\phi_\eta$ is strictly convex on $[1-\eta,+\infty)$, $\phi_\eta([0,1-\eta])=\set{0}$ and $\phi_\eta(1)=1$.
\item{(ii)} If $0<\eta_1\leq \eta_2 <1$ then $\phi_{\eta_1}(x)\leq \phi_{\eta_2}(x)$ for any $x \in [0,1]$.  
\end{list}
One example of such a system can be constructed as follows: let $\phi:\Real \to \Real$ be \smooth\ such that $\phi(x)=0$ if $x\leq 0$, $\phi(1)=1$ and $\phi$ is increasing and strictly convex on $[0,+\infty)$. We define $\phi_\eta(x):=\phi(\frac{x-(1-\eta)}{\eta})$ for all $x \in [0,1]$. Now the system $\set{\phi_\eta}$ satisfies (ii) since $\eta \mapsto \frac{x-(1-\eta)}{\eta}$ is increasing for every $x \in [0,1)$ while the validity of (i) follows from properties of $\phi$.

We define a function $\Phi_\eta:\linf(\Gamma) \to \ereal$ by
\[
\Phi_\eta(x)= \sum_{\gamma \in \Gamma} \phi_\eta(\abs{x(\gamma)}).
\]

Let us define $Z_\eta:\linf(\Gamma) \to \Real$ as the Minkowski functional of the set $C=\set{x\in\linf(\Gamma); \Phi_\eta(x)< 1/2}$.

\begin{lem}\label{l:FetaDef}
Let $0<\eta<1$ be fixed.
Then $Z_\eta$ is a strongly lattice seminorm such that $(1-\eta)Z_\eta(x)\leq \norm{x}_\infty$ and $Z_\eta$ is LFC, \smooth\ and strictly positive in the set \[A_\eta(\Gamma):=\{x \in \linf(\Gamma):
  \finitetop{x}<(1-\eta)\norm{x}_\infty\}.\]
Moreover $(1-\eta)Z_\eta(x)< \norm{x}_\infty$ for all $x\in A_\eta(\Gamma)$.
\end{lem}
\begin{proof}
The set $C$ is symmetric convex with zero as interior point (indeed, $(1-\eta)\closedball{\linf(\Gamma)} \subset C$) so $Z_\eta$ is $\frac{1}{1-\eta}$-Lipschitz and convex.

Let $A_\eta'(\Gamma):=\set{x\in \linf(\Gamma):\finitetop{x}<1-\eta}$. This set is convex and open since $\finitetop{\cdot}$ is a continuous and convex (seminorm).
The function $\Phi_\eta$ is in $A_\eta'(\Gamma)$ a locally finite sum of convex \smooth\  functions, thus it is a convex function which is  LCF and \smooth\ in $A_\eta'(\Gamma)$.

Let us fix $x_0 \in A'_\eta(\Gamma)$
such that $\Phi_\eta(x_0)=1/2$. Then, since $\phi_\eta$ is increasing at the points where it is not zero, we get $Z_\eta(x_0)=1$ and $\Phi_\eta'(x_0)x_0>0$. As is usual, we consider the equation $\Phi_\eta\left(\frac{x}{Z_\eta(x)}\right)=\frac{1}{2}$. By the Implicit Function Theorem, this equation locally redefines $Z_\eta$ and proves that $Z_\eta$ is \smooth\ on some neighborhood $U$ of $x_0$ since $\Phi_\eta$ is. Moreover by application of Lemma~\ref{l:implicitLFC} we get that $Z_\eta$ is LFC at $x_0$.

To prove that $Z_\eta$ is LFC, strictly positive and \smooth\ in $A_\eta(\Gamma)$ it is enough to show that for each $x\in A_\eta(\Gamma)$ there is $\lambda >0$ such that $\lambda x \in A'_\eta(\Gamma)$ and $\Phi_\eta(\lambda \cdot x)=1/2$ and then use the homogeneity of $Z_\eta$.

Let $x \in A_\eta(\Gamma)$.
Then $\finitetop{\frac{x}{\norm{x}_\infty}}<1-\eta$ and since $A_\eta'(\Gamma)$ is convex, it follows that $[0,\frac{x}{\norm{x}_\infty}] \subset  A_\eta'(\Gamma)$.
We have for such $x$ that $\Phi_\eta(\frac{x}{\norm{x}_\infty})\geq 1$, $\Phi_\eta(0\cdot x)=0$ and the mapping $\lambda \mapsto \Phi_\eta(\lambda x)$ is continuous for $\lambda \in [0,\frac{1}{\norm{x}_\infty}]$. Hence there must exist $\lambda \in (0,\frac{1}{\norm{x}_\infty})$ such that $\lambda x \in A_\eta'(\Gamma)$ and $\Phi_\eta(\lambda \cdot x)=1/2$. 

We continue showing that $Z_\eta$ is strongly lattice. First observe that $\Phi_\eta$ is strongly lattice as $\phi_\eta$ is nondecreasing.
Let $\abs{x} \leq \abs{y}$ and $Z_\eta(x)=1$. Then $x\in \partial C$ which implies that $\finitetop{x}=1-\eta$ or $\Phi_\eta(x)=1/2$. Since both functions $\finitetop{\cdot}$ and $\Phi_\eta$ are strongly lattice, we conclude that $\finitetop{y}\geq 1-\eta$ or $\Phi_\eta(y)\geq 1/2$ which in turn implies that $Z_\eta(y)\geq 1$. For a general $x$ we employ the homogeneity of $Z_\eta$, so $Z_\eta$ is strongly lattice.

Finally, if $x \in A_\eta(\Gamma)$, then the above considerations imply that $\Phi_\eta\left(\frac{x}{Z_\eta(x)}\right)=1/2$. This is possible only if there is some $\gamma \in \Gamma$ such that $\frac{x(\gamma)}{Z_\eta(x)}>1-\eta$, and the moreover claim follows.
\end{proof}

\begin{lem}\label{l:monotonicity}
Let $0<\eta_1\leq \eta_2<1$. Then $Z_{\eta_1}(x)\leq Z_{\eta_2}(x)$ for every $x \in A_{\eta_2}(\Gamma)$.
\end{lem}
\begin{proof}
First of all, if $x \in A_{\eta_2}(\Gamma)$, then $x \in A_{\eta_1}(\Gamma)$. So the equivalence $Z_{\eta_i}(\lambda x)=1 \Leftrightarrow \Phi_{\eta_i}(\lambda x)=1/2$ holds for both $i=1,2$. Let us assume that $Z_{\eta_1}(\lambda x)=1$ for some $\lambda>0$. Then the ordering of functions $\phi_\eta$ yields $1/2=\Phi_{\eta_1}(\lambda x) \leq \Phi_{\eta_2}(\lambda x)$ which results in $Z_{\eta_2}(\lambda x)\geq 1$. 
\end{proof}

\begin{lem}\label{l:PWZ}
Let $0<\eta<1$ be given and let $x_\kindex, x \in A_\eta(\Gamma)$ ($\kindex \in \Natural$) be non-negative (in the lattice $\linf(\Gamma)$) such that
\[
2Z^2_\eta(x)+2Z^2_\eta(x_\kindex)-Z^2_\eta(x+x_\kindex)\to 0 \mbox{ as } \kindex \to \infty.
\]
Then $x_\kindex(\gamma) \to x(\gamma)$ for any $\gamma \in \Gamma$ such that $x(\gamma) > Z_\eta(x)(1-\eta)$.
\end{lem}
\begin{proof}
The assumption and Lemma~\ref{l:equivLUR} yield
\begin{equation}\label{e:normToNorm}
Z_\eta(x_\kindex) \to Z_\eta(x) \quad \mbox{ and }\quad Z_\eta\left(\frac{x+x_\kindex}{2}\right)\to Z_\eta(x).
\end{equation}
Let us put $\tilde{x}:=\frac{x}{Z_\eta(x)}$ and $\tilde{x}_\kindex:=\frac{x_\kindex}{Z_\eta(x_\kindex)}$.
We get from \eqref{e:normToNorm} that
\[
2Z_\eta^2(\tilde{x})+2Z_\eta^2(\tilde{x}_\kindex)-Z_\eta^2(\tilde{x}+\tilde{x}_\kindex) \to 0.
\]
Since $Z_\eta(\tilde{x})=Z_\eta(\tilde{x}_\kindex)=1$, the above implies that
\[
\lambda_\kindex:=Z_\eta(\tilde{x}+\tilde{x}_\kindex) \to 2.
\]
We may deduce from $x, x_\kindex \in A_\eta(\Gamma)$ that $\Phi_\eta(\tilde{x})=1/2=\Phi_\eta(\tilde{x}_\kindex)$ for all $k \in  \Natural$.
Also, $\Phi_\eta(\lambda_\kindex^{-1}(\tilde{x}+\tilde{x}_\kindex))=1/2$ for all but finitely many $k \in \Natural$. Indeed, if $\Phi_\eta(\lambda_\kindex^{-1}(\tilde{x}+\tilde{x}_\kindex))\neq 1/2$, then
$\lambda_\kindex^{-1}(\tilde{x}+\tilde{x}_\kindex) \in \partial A'_\eta(\Gamma)$.
Then in fact $\finitetop{\lambda_\kindex^{-1}(\tilde{x}+\tilde{x}_\kindex)}=1-\eta$.
As $\tilde{x}\in A'_\eta(\Gamma)$, there is $\xi >0$ such that $\finitetop{\tilde{x}}+\xi<1-\eta$. By the same reasoning $\finitetop{\tilde{x}_\kindex}<1-\eta$. By the convexity (subaditivity) of $\finitetop{\cdot}$ and these estimates one has
\[
\finitetop{\tilde{x}+\tilde{x}_\kindex} \leq \finitetop{\tilde{x}}+\finitetop{\tilde{x}_\kindex}<2(1-\eta) - \xi.
\]
Finally, $\lambda_\kindex< \frac{2(1-\eta)-\xi}{1-\eta}$ which can happen only for finitely many $\kindex$ as $\lambda_\kindex \to 2$.

As $\Phi_\eta$ is continuous at $\tilde{x}$ and $\lambda_\kindex \to 2$, it follows
\[
\Phi_\eta((\lambda_\kindex-1)^{-1}\tilde{x}) \to 1/2.
\]
Consequently
\begin{equation}\label{e:scaledConvexity}
(1-\lambda_\kindex^{-1})\Phi_\eta\left((\lambda_\kindex-1)^{-1}\tilde{x}\right)+\lambda_\kindex^{-1} \Phi_\eta(\tilde{x}_\kindex)-\Phi_\eta\left(\lambda_\kindex^{-1}(\tilde{x}+\tilde{x}_\kindex)\right) \to 0.
\end{equation}
Let $a>1-\eta$. The definition of $\phi_\eta$ and a compactness argument imply that for each $\varepsilon>0$ there exists $\Delta>0$ such that if for reals $r, s, \alpha$ it holds
\begin{itemize}
\item $0\leq r \leq 4\max\set{\norm{\tilde{x}}_\infty, \sup_\kindex\norm{\tilde{x}_\kindex}_\infty}$,
\item $a\leq s \leq 4\max\set{\norm{\tilde{x}}_\infty, \sup_\kindex\norm{\tilde{x}_\kindex}_\infty}$,
\item $\frac{1}{4}\leq \alpha \leq \frac{3}{4}$, and
\item $\alpha \phi_\eta(r)+(1-\alpha)\phi_\eta(s)-\phi_\eta(\alpha r + (1-\alpha) s)<\Delta$,
\end{itemize}
then $\abs{r-s}<\varepsilon$.

In particular, let $a>1-\eta$ be such that $\set{\gamma \in \Gamma; \tilde{x}(\gamma)>1-\eta}=\set{\gamma \in \Gamma; \tilde{x}(\gamma)>a}$ and let $\gamma \in \Gamma$ be such that $\tilde{x}(\gamma)>a$. Then for $\kindex$ large enough we have $(\lambda_\kindex-1)^{-1}\tilde{x}(\gamma)>a$
so we may substitute $r:=\tilde{x}_\kindex(\gamma)$, $s:=(\lambda_\kindex-1)^{-1}\tilde{x}(\gamma)$ and $\alpha:=\lambda_\kindex^{-1}.$
It follows from \eqref{e:scaledConvexity} that one has $\abs{(\lambda_\kindex-1)^{-1}\tilde{x}(\gamma)-\tilde{x}_\kindex(\gamma)}\to 0$ as $k \to \infty$. Since $\lambda_\kindex \to 2$ and using \eqref{e:scaledConvexity}, we finally get that $x_\kindex(\gamma)\to x(\gamma)$ as $k\to \infty$.
\end{proof}

The following system of convex functions is at the heart of our construction.
We recall that $C\geq 1$ is the constant of equivalence between the norms $\abs{\cdot}$ and $\norm{\cdot}$, which was introduced in Section~\ref{s:AboutN}.
\begin{lem}\label{l:gsystem}
There exist
\begin{itemize}
\item a decreasing sequence of positive numbers $\delta_n \searrow 0$; $\delta_1<2C$;
\item a decreasing sequence of positive numbers $\rho_n \searrow 0$;
\item positive numbers $\kappa_{n,m}>0$ such that for each $n \in \Natural$ the sequence $(\kappa_{n,m})_m$ is decreasing and $\kappa_{n,m} \too{m} 0$; for each $n,m \in \Natural$ one has $\rho_n> 2 \kappa_{n,m}$;
\item an equi-Lipschitz system of non-negative, \smooth, $1$-bounded, convex functions
\[
\set{g_{n,m,l}:\domain_{n,l} \to \Real:n,m \in \Natural, l=1,\ldots, n},
\]
where $\domain_{n,l}:=[0,2nC-\delta_n (n-l)]\times[0,1+2nC]$,
satisfying (with $n,m,l \in \Natural$, $l\leq n$, resp. $l<n$ in (A2),(A5))
\begin{list}{}{}
\item{\rm (A1)} $g_{n,m,l}(t,s)=0$ iff $(t,s) \in [0,l\delta_n]\times [0,1+2nC]=:\nullset_{n,l}$;
\item{\rm (A2)} $g_{n,m,l}(t,s)\geq g_{n,m,l+1}(t,s)+\rho_n$
whenever $(t,s) \in \domain_{n,l} \setminus \nullset_{n,l+1}$;
\item{\rm (A3)} if $(t,0) \in \domain_{n,l} \setminus \nullset_{n,l}$, then $s \mapsto g_{n,m,l}(t,s)$ is increasing on $[0,1+2nC]$ and \[g_{n,m,l}(t,1+2nC)-g_{n,m,l}(t,0)\leq\kappa_{n,m};\]
\item{\rm (A4)} if $(t,0) \in \domain_{n,l}$, then $g_{n,m,l}(t,0)=g_{n,m+1,l}(t,0)$;
\item{\rm (A5)} for all $(t,s) \in \domain_{n,l}\setminus \nullset_{n,l}$ it holds $g_{n,m,l}(t,s)<g_{n,m,l+1}(t+r,s)$ provided $r>\delta_n$.
\item{\rm (A6)} Let $(t,s) \in \domain_{n,l} \setminus \nullset_{n,l}$. If $(t_\kindex,s_\kindex) \in \domain_{n,l}$ and $t_\kindex \to t$ and $g_{n,m,l}(t_\kindex,s_\kindex) \to g_{n,m,l}(t,s)$ as $\kindex \to \infty$, then $s_\kindex \to s$.
\item{\rm (A7)} The mapping $(t,s) \mapsto g_{n,m,l}(\abs{t},\abs{s})$ is strongly lattice in $\domain_{n,l}$.
\end{list}
\end{itemize}
\end{lem}
\begin{proof}
Let $f:\Real \to [0,+\infty)$ be defined as
\[
f(t):=
\begin{cases}
0, \mbox{ for } t\leq 0,\\
\exp(-\frac{1}{t^2}) \mbox{ for } t>0.
\end{cases}
\]
It is elementary (one may use Lemma~\ref{l:sufficientConvexity}) to check that $f(t)\cdot (s^2+s+1)$ is convex in the strip $(-\infty,10^{-1}]\times[0,10^{-1}]$
so the function
\[
g(t,s):=f(10^{-1}t)\cdot((10^{-1}s)^2+10^{-1}s+1)
\]
is convex in the strip $(-\infty,1]\times[0,1]$. We take for $(\delta_n)_n$ just any decreasing null sequence of positive
numbers such that $\delta_1<2C$, and we  define
\[
g_{n,m,l}(t,s):=g\left(\frac{t-\delta_n l}{(2C-\delta_n)n},\theta_{n,m}\frac{s}{1+2nC}\right).
\]
where $\theta_{n,m}\in (0,1)$ will be chosen later.
Now since our functions $g_{n,m,l}$ are just shifts and stretches of one non-negative, \smooth, $1$-bounded, Lipschitz, convex function, it follows that all $g_{n,m,l}$ share these properties (with the same Lipschitz constant).

Properties (A1), (A4) and (A5) are straightforward, see also Figure~\ref{f:system}.
\begin{figure}[h]\label{f:system}
\centering
\includegraphics{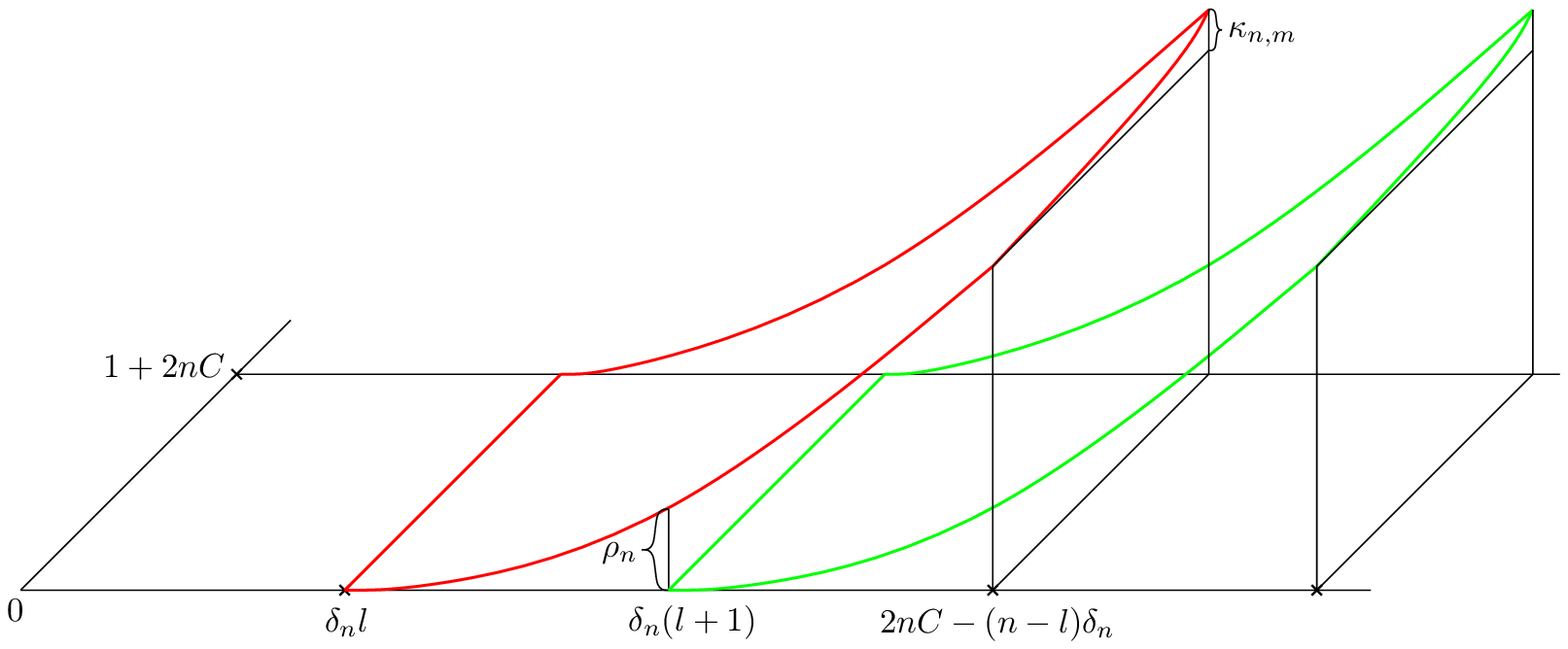}
\end{figure}
Notice that, when $t>0$, the function $s \mapsto g(t,s)$ is increasing on $[0,1]$. This implies the first part of (A3).
In order to satisfy (A2), we may define $\rho_n$ as
\[
\rho_n:=\inf\set{g_{n,m,l}(t,s)-g_{n,m,l+1}(t,s): l,m \in \Natural, l<n, (t,s) \in \domain_{n,l}\setminus \nullset_{n,l+1}}
\]
which evaluates as $\rho_n=g_{n,1,1}(2\delta_n,0)= f\left(\frac{\delta_n}{(2-\delta_n)n}\right)\searrow 0$ as $n \to \infty$. Notice that this $\rho_n$ does not depend on the choice of $\theta_{n,m}$.
On the other hand, in order to fulfill (A3), $\kappa_{n,m}$ may be defined as
\[
\kappa_{n,m}:=\sup\set{g_{n,m,l}(t,1+2nC)-g_{n,m,l}(t,0): l\leq n, (t,0) \in \domain_{n,l}}
\]
which evaluates as $\kappa_{n,m}=g_{n,m,n}(2nC,1+nC)-g_{n,m,n}(2nC,0)$. We see that, by an appropriate choice of $\theta_{n,m}$ (in particular, for each $n \in \Natural$, the sequence $(\theta_{n,m})_m$ should be decreasing to zero), one may satisfy the requirements $\rho_n>2 \kappa_{n,m}$ and $\kappa_{n,m} \searrow 0$ as $m \to \infty$.

For the proof of (A6) let us assume that $s_\kindex \nrightarrow s$. The fact that $g_{n,m,l}(t_\kindex,\cdot) \to g_{n,m,l}(t,\cdot)$ uniformly on $[0,1+2nC]$ leads quickly to a contradiction.

Finally (A7) follows since $g$ is non-decreasing in $\domain_{n,l}$ in each variable.
\end{proof}
Let us fix, for each $\delta>0$, some \smooth, convex mapping $\xi_\delta$ from $[0,+\infty)$ to $[0,+\infty)$ which satisfies $\xi_\delta([0,\delta])=\set{0}$, $\xi_\delta(t)>0$ for $t>\delta$ and $\xi_\delta(t)=t-2\delta$ for $t\geq 3\delta$. Such a mapping can be constructed e.g. by integrating twice a \smooth, non-negative bump.

\begin{lem}\label{l:systemH}
Let $n,m \in \Natural$ be fixed and let us define a mapping $H_{n,m}:\openball{(\Banach,\norm{\cdot})} \to \linf(F_n)$ where $F_n=\set{(A,B) \in 2^\ointerval\times 2^\ointerval: \abs{A} \leq n, B \subset A, A \neq \emptyset \neq B}$ by
\[
H_{n,m}x(A,B):=g_{n,m,\abs{A}}\left(\sum_{\gamma \in A} \xi_{\delta_n}(\norm{Q_\gamma x}_\gamma),\xi_{\delta_n}( \norm{P_B x-x})\right).
\]
Then $H_{n,m}$ is a continuous, coordinatewise convex and coordinatewise \COne\ mapping, and for each $x\in X$ such that $\norm{x}< 1$ it holds $H_{n,m}x \in A_{\rho_n/2-\kappa_{n,m}}(F_n)\cup \set{0}$ (see the definition of the set $A_{\rho_n/2-\kappa_{n,m}}(F_n)$ in Lemma~\ref{l:FetaDef}).
\end{lem}
Notice that, by the definition of $\kappa_{n,m}$ in Lemma~\ref{l:gsystem}, we have always $\rho_n/2-\kappa_{n,m}>0$. We will use the notation $\eta_{n,m}:=\rho_n/2-\kappa_{n,m}$.
\begin{proof}
When $\norm{x}<1$, then \eqref{e:inductiveHypo} yields $\left(\sum_{\gamma \in A} \xi_{\delta_n}(\norm{Q_\gamma x}_\gamma),\xi_{\delta_n}( \norm{P_B x-x})\right) \in [0,1+2\cardinality{A}C)\times[0,2\cardinality{A}C)\subset \domain_{n,\cardinality{A}}$. So for each $(A,B) \in F_n$ the mapping $x \mapsto H_{n,m}x(A,B)$ is \COne\ as a composition of such mappings. Also, $\set{x \mapsto H_{n,m}x(A,B):(A,B) \in F_n}$ is equi-Lipschitz thus $H_{n,m}$ is continuous. Each $x \mapsto H_{n,m}x(A,B)$ is convex by application of Lemma~\ref{l:SLconvexComposition} since $g_{n,m,l}$ is convex and strongly lattice. Because $\sup g_{n,m,l}(\domain_{n,\cardinality{A}})<1$ for each $l \leq n$, we get that $\norm{H_{n,m}x}_\infty < 1$.

We are going to prove that $\finitetop{H_{n,m}x}<\norm{H_{n,m}x}_\infty(1-\rho_n/2+\kappa_{n,m})$ or $\norm{H_{n,m}x}_\infty=0$.
For any $x \in X$ and $\delta>0$, let $\ointerval(x,\delta):=\set{\gamma \in \ointerval: \norm{Q_\gamma x}_\gamma>\delta}$.
Let $x \in \openball{(\Banach,\norm{\cdot})}$ be fixed and let us
define a set $E \subset F_n$ as $E:=\set{(A,B) \in F_n:A \subset \ointerval(x,\delta_n)}$. Since $E$ is finite, it holds
\begin{equation}\label{e:essSupH}
\finitetop{H_{n,m}x}=\finitetop{H_{n,m}x\upharpoonright_{F_n\setminus E}}\leq \sup \set{H_{n,m}x(A,B):(A,B) \in F_n \setminus E}.
\end{equation}
If there is no $(A,B) \in F_n \setminus E$ such that $H_{n,m}(A,B)>0$, then $\finitetop{H_{n,m}x}=0$ and our claim is trivially true. We proceed assuming that $H_{n,m}x(A,B)>0$ for some $(A,B) \in F_n \setminus E$. Then
\[
\left(\sum_{\gamma \in A}\xi_{\delta_n}(\norm{Q_\gamma x}_\gamma),\xi_{\delta_n}(\norm{P_B x-x})\right) \notin \nullset_{n,\abs{A}}
\]
which, by (A1) in Lemma~\ref{l:gsystem}, can happen only if $C:=A \cap \ointerval(x,\delta_n)\neq \emptyset$. Since $(A,B) \notin E$, we have $\cardinality{C}<\cardinality{A}$. It follows from Lemma~\ref{l:gsystem} (A2) and (A3) that
\[
\begin{split}
g_{n,m,\abs{A}}\left(\sum_{\gamma \in A} \xi_{\delta_n}(\norm{Q_\gamma x}_\gamma),\xi_{\delta_n}(\norm{P_B x-x})\right)
&\leq g_{n,m,\abs{C}}\left(\sum_{\gamma \in C} \xi_{\delta_n}(\norm{Q_\gamma x}_\gamma),\xi_{\delta_n}(\norm{P_B x-x})\right)-\rho_n\\
&\leq g_{n,m,\abs{C}}\left(\sum_{\gamma \in C} \xi_{\delta_n}(\norm{Q_\gamma x}_\gamma),\xi_{\delta_n}(\norm{P_D x-x})\right)-\rho_n+\kappa_{n,m}
\end{split}
\]
for any $D \subset C$. Of course, since $\ointerval(x,\delta_n)$ is finite, there are only finitely many couples $(C,D)$ such that $D\subset C \subset \ointerval(x,\delta_n)$.
We may therefore write
\[
H_{n,m}x(A,B)\leq \max_{D\subset C \subset \ointerval(x,\delta_n)} H_{n,m}x(C,D)-\rho_n+\kappa_{n,m}\leq\norm{H_{n,m}}_\infty(1-\rho_n+\kappa_{n,m})
\]
for any $(A,B) \in F_n \setminus E$.
This together with \eqref{e:essSupH} gives
$\finitetop{H_{n,m}x} < \norm{H_{n,m}}_\infty(1-(\rho_n/2-\kappa_{n,m})).$
\end{proof}

\begin{lem}\label{l:TroyanskiElement}
Let $0\neq x \in \openball{(\Banach,\norm{\cdot})}$ and let $A$ be a finite subset of $\ointerval$ such that $Q_\gamma x \neq 0$ when $\gamma \in A$.
We claim that, for all $n,m \in \Natural$ sufficiently large, there exists a finite $C_{n,m} \subset \ointerval$ such that
\begin{itemize}
 \item $A \subset C_{n,m}$, and
 \item $H_{n,m}x(C_{n,m},A)>(1-\eta_{n,m})Z_{\eta_{n,m}}(H_{n,m}x)$.
\end{itemize}
\end{lem}

\begin{proof}
We start by defining $A^*:=\set{\gamma \in \ointerval: \norm{Q_\gamma x}_\gamma \geq \min_{\alpha \in A} \norm{Q_\alpha x}_\alpha}$ and we set out for finding  $C_{n,m}$ so that in fact $A^* \subset C_{n,m}$.

Let us investigate the mapping $L_n:\openball{(X,\norm{\cdot})} \to \linf(F_n)$ defined as
\[
L_n y(D,E):=g_{n,1,\cardinality{D}}\left(\sum_{\gamma \in D} \xi_{\delta_n}(\norm{Q_\gamma y}_\gamma),0\right).
\]
By the same argument as in the proof of Lemma~\ref{l:systemH}, we get that
$\finitetop{L_n x}\leq (1-\rho_n) \norm{L_n x}_\infty$ or $L_n x=0$. Hence $L_n x \in A_{\rho_n/2} \cup \set{0}$.
 If $n$ is large enough, necessarily $L_n x \neq 0$. It follows that $L_n x$ attains a nonzero maximum.
For $n \in \Natural$, let $C_{n}$ be such that $L_{n}x(C_{n},D)=\norm{L_{n}x}_\infty$ for some (and all) non-empty $D \subset C_{n}$. We claim that, for $n$ sufficiently large, $A^* \subset C_{n}$.

Let us denote $b:=\min\set{\norm{Q_\gamma x}_\gamma: \gamma \in A^*} - \max\set{\norm{Q_\gamma x}_\gamma: \gamma \in \ointerval \setminus A^*}$. Since $Q_\gamma x \neq 0$ for all $\gamma \in A$, and for the $c_0$-nature of $(\norm{Q_\gamma x}_\gamma)_{\gamma \in \ointerval}$, it follows that $b>0$.
Notice that
\[
b_n:=\xi_{\delta_n}\left(\min\set{\norm{Q_\gamma x}_\gamma: \gamma \in A^*}\right) - \xi_{\delta_n}\left(\max\set{\norm{Q_\gamma x}_\gamma: \gamma \in \ointerval \setminus A^*}\right) \to b \mbox{ as } n \to \infty.
\]
Let $n\geq \cardinality{A^*}$ be so large that $\delta_n < \xi_{\delta_n}\left(\min\set{\norm{Q_\gamma x}_\gamma: \gamma \in A^*}\right)$ and $\delta_n < b_n$.

If $A^* \nsubseteq C_{n}$, there exists $\gamma_1 \in A^* \setminus C_{n}$. If $\cardinality{C_{n}}<n$, then we define $\tilde{C}_{n}:=\set{\gamma_1} \cup C_{n}$.
By our choice of $n$, we have that $\xi_{\delta_n}(\norm{Q_{\gamma_1}x}_{\gamma_1})>\delta_n$ and so by the property (A5) in Lemma~\ref{l:gsystem} we get that
\[
g_{n,1,\cardinality{C_{n}}}\left(\sum_{\gamma \in C_{n}} \xi_{\delta_n}( \norm{Q_\gamma x}_\gamma),0\right) < g_{n,1,\cardinality{\tilde{C}_{n}}}\left(\sum_{\gamma \in \tilde{C}_n} \xi_{\delta_n}( \norm{Q_\gamma x}_\gamma),0\right)
\]
contradicting that any couple $(C_{n},D)\in F_n$ maximizes $L_{n}x$.

If $\cardinality{C_{n}}=n$, then there exists $\gamma_2 \in C_{n} \setminus A^*$ and we define $\tilde{C}_{n}:=\set{\gamma_1}\cup C_{n}\setminus {\gamma_2}$. Our choice of $n$ yields that $\xi_{\delta_n}(\norm{Q_{\gamma_1}x}_{\gamma_1})-\xi_{\delta_n} (\norm{Q_{\gamma_2}x}_{\gamma_2})>\delta_n$ so (A5) in Lemma~\ref{l:gsystem} implies
\[
g_{n,1,n}\left(\sum_{\gamma \in C_{n}} \xi_{\delta_n}( \norm{Q_\gamma x}_\gamma),0\right) < g_{n,1,n}\left(\sum_{\gamma \in \tilde{C}_{n}} \xi_{\delta_n}( \norm{Q_\gamma x}_\gamma),0\right)\\
\]
once again contradicting that any couple $(C_{n},D)\in F_n$ maximizes $L_{n}x$. So $A^* \subset C_n$.

At this moment, we leave $n$ fixed according to the choices above and we start tuning $m$. First of all, let us observe that $L_n x(C_n,A)>Z_{\rho_n/2}(L_n x)(1-\rho_n/2)$ by the moreover part of Lemma~\ref{l:FetaDef}. Since $\eta_{n,m} \nearrow \rho_n/2$ as $m \to \infty$, we deduce that there is some $p\in \Natural$ such that $L_n x(C_n,A)>Z_{\rho_n/2}(L_n x)(1-\eta_{n,p})$.
We will work, for $\gamma \in F_n$, with the set $M_\gamma=\set{u \in \linf(F_n): \abs{u(\gamma)} > Z_{\rho_n/2}(u)(1-\eta_{n,p})}$. The set $M_\gamma$ is open and, in particular, $L_n x \in M_{(C_n,A)}$.

Using (A3) and (A4) in Lemma~\ref{l:gsystem} we may see that $H_{n,m}x \to L_n x$ in $(\linf(F_n),\norm{\cdot}_\infty)$ as $m \to \infty$. Since $L_n x$ is a member of the open set $A_{\rho_n/2}(F_n)$, so will be $H_{n,m}x$ for $m$ large enough. Similarly, the openness of $M_{(C_n,A)}$ insures that  $H_{n,m}x \in M_{(C_n,A)}$ for $m\geq p$ and large enough. This means that
\[
H_{n,m}x(C_n,A) >Z_{\rho_n/2}(H_{n,m}x)(1-\eta_{n,p})
\geq Z_{\eta_{n,m}}(H_{n,m}x)(1-\eta_{n,m})
\]
where the second inequality follows from Lemma~\ref{l:monotonicity} as $\rho_n/2 \geq \eta_{n,m}$ and $\eta_{n,m}\geq \eta_{n,p}$ for all $m \geq p$.
So we may define $C_{n,m}:=C_n$ for $m$ sufficiently large.
\end{proof}
We came close to the definition of the norm $J$.
First, we choose some decreasing sequence of positive numbers $\sigma_j\searrow 0$ and we define
$J_{n,m}:\openball{(\Banach,\norm{\cdot})} \to \Real$ as
\[
J_{j,n,m}(x):=\xi_{\sigma_j}(Z_{\eta_{n,m}}(H_{n,m}x)).
\]
Next, let $\tilde{J}:\openball{(\Banach,\norm{\cdot})} \to \Real$ be defined as
\[\label{e:tildeJdef}
\tilde{J}^2(x):=\norm{x}^2+\sum_{j,n,m  \in \Natural} \frac{1}{2^{j+n+m}} J_{j,n,m}^2(x)
\]
and finally let $J:\Banach \to \Real$ be defined as the Minkowski functional of $\set{x \in X: \tilde{J}(x) \leq 1/2}$.
\begin{lem}\label{l:JisEquivNorm}
The function $J$ is an equivalent norm on $\Banach$ which is \COne\ away from the origin.
\end{lem}
\begin{proof}
%
By the Implicit Function Theorem, to see the differentiability it is sufficient to show that for each $x \in \Banach$ such that $J(x)=1$, the function $\tilde{J}$ is Fr\'echet differentiable on some neighborhood of $x$ with $\tilde{J}'(x)x\neq 0$.

First of all let us observe that, for each $\sigma>0$ and $\eta>0$, the composed function $\xi_\sigma \circ Z_\eta: \linf(F_n) \to \Real$ is \smooth\ and LFC in $A_\eta(F_n)\cup \set{0}$. Of course it is -- we know it already for points in $A_\eta(F_n)$ and clearly, there is a neighborhood $U$ of $0\in F_n$ such that $\xi_\sigma \circ Z_\eta$ is constant in $U$.

Let $x \in \Banach$ such that $J(x)=1$. Then $\norm{x}\leq 1/2$.
It follows from Lemma~\ref{l:systemH} and from Lemma~\ref{l:composedLFC} that
each $J_{j,n,m}$ is \COne\ at $x$.
Further we claim that there is a constant $K>0$ such that each $J_{j,n,m}$ is $nK$-Lipschitz. Indeed, there is a constant $K'>0$ such that $H_{n,m}$ is $(1+2nC)K'$-Lipschitz for all $n,m \in \Natural$; $Z_{\eta}$ is $2$-Lipschitz for each $0<\eta<1/2$ and $\xi_\sigma$ is $1$-Lipschitz for each $\sigma>0$.  It follows that $\tilde{J}$ is $K''$-Lipschitz for some $K''>0$.
The calculus rules lead to the conclusion that $\tilde{J}$ is Fr\'echet differentiable on a neighborhood of any $x\in \Banach$ such that $\norm{x}< 1$; then the convexity of all terms implies that $\tilde{J}'(x)x >  0$.

Finally, $2\norm{x}\leq J(x) \leq 2K'' \norm{x}$ where the second inequality follows from the $K''$-Lipschitzness of $\tilde{J}$.
\end{proof}

\begin{lem}\label{l:finalHit}
Whenever $x_\kindex,x \in \Banach$, $\kindex \in \Natural$, are such that \eqref{e:LURmain} holds,
then for each $\varepsilon>0$ there is a finite subset $A$ of $\ointerval$ such that $Q_\gamma x \neq 0$ for $\gamma \in A$, $\norm{P_A x-x}<\varepsilon$ and $\norm{P_A x_\kindex-x_\kindex}<\varepsilon$ for all $\kindex$ sufficiently large.
\end{lem}

\begin{proof}
We may assume, that $J(x)=1$.
We start by finding a finite $A \subset \ointerval$ such that $\norm{P_A x-x}<\varepsilon/2$ and such that $Q_\gamma x \neq 0$ for $\gamma \in A$. This is possible by Lemma~\ref{l:PRI}. Now we will just show that $\norm{P_A x_\kindex - x_\kindex} \to \norm{P_A x-x}$.

It follows from \eqref{e:LURmain} and from the uniform continuity of $\tilde{J}$ on bounded sets that
\begin{equation}\label{e:LURJ}
\frac{\tilde{J}^2(x_\kindex)+\tilde{J}^2(x)}{2}-\tilde{J}^2\left(\frac{x+x_\kindex}{2}\right) \to 0 \mbox{ as } k\to \infty.
\end{equation}
By the convexity of the terms in the definition of $\tilde{J}$, we get that
\begin{equation}\label{e:LURJsub}
\frac{J^2_{j,n,m}(x_\kindex)+J^2_{j,n,m}(x)}{2} - J^2_{j,n,m}\left(\frac{x+x_\kindex}{2}\right) \to 0 \mbox{ as } k\to \infty.
\end{equation}
for each $j,n,m \in \Natural$.

Let us borrow the notation $L_n x$ from the proof of Lemma~\ref{l:TroyanskiElement}. Let us recall that $H_{n,m}x \geq L_n x\geq 0$ (in the lattice $\linf(F_n)$) for all $m \in \Natural$.
There is some $n_0 \in \Natural$ such that for all $n\geq n_0$ we have that $L_{n}x \neq 0$. Hence
$Z_{\eta_{n,m}}(H_{n,m}x) \geq Z_{\eta_{n,m}}(L_n x) \geq Z_{\eta_{n,1}}(L_n x)>0$ for $n\geq n_n$ and $m \in \Natural$. Therefore for each $n \geq n_0$ there exists  $j_n \in \Natural$ such that for all $j \geq j_n$ and all $m \in \Natural$ one has
$J_{j,n,m}> 0$. Since, for $n\geq n_0, m \in \Natural$ and $j\geq j_n$, \eqref{e:LURJsub} is equivalent to
\[
\lim_\kindex J_{j,n,m}(x_\kindex)=J_{j,n,m}(x)= \lim_\kindex J_{j,n,m}\left(\frac{x_\kindex+x}{2}\right)
\]
and $\xi_\sigma\upharpoonright_{(\sigma,+\infty)}$ has a continuous inverse,
it follows that
\[
\frac{Z_{\eta_{n,m}}^2(H_{n,m}x_\kindex)+ Z_{\eta_{n,m}}^2(H_{n,m}x)}{2} - Z_{\eta_{n,m}}^2\left(H_{n,m}\left(\frac{x_\kindex+x}{2}\right)\right) \too{\kindex} 0
\]
for all $n\geq n_0$ and $m \in \Natural$.
Since $x \mapsto H_{n,m}x(A,B)$ is convex and non-negative for each $(A,B) \in F_n$ and since $Z_{\eta_{n,m}}$ is strongly lattice and convex it follows
\[
\begin{split}
0 \leftarrow \frac{Z_{\eta_{n,m}}^2(H_{n,m}x_\kindex)+ Z_{\eta_{n,m}}^2(H_{n,m}x)}{2} - Z_{\eta_{n,m}}^2\left(H_{n,m}\left(\frac{x_\kindex+x}{2}\right)\right) &\geq \\
\geq \frac{Z_{\eta_{n,m}}^2(H_{n,m}x_\kindex)+Z_{\eta_{n,m}}^2(H_{n,m}x)}{2} - Z_{\eta_{n,m}}^2\left(\frac{H_{n,m}x_\kindex+H_{n,m}x}{2}\right) &\geq 0
\end{split}
\]
for every $n \geq n_0$ and $m \in \Natural$.
Let us fix $n \geq n_0$ and $m \in \Natural$ both large enough in the sense of Lemma~\ref{l:TroyanskiElement}. We also require that $\delta_n <\norm{P_A x-x}$. By application of Lemma~\ref{l:TroyanskiElement}, we obtain a set $C_{n,m}$ such that $\gamma:=(C_{n,m},A)\in F_n$ satisfies the assumptions of Lemma~\ref{l:PWZ}. Thus, using this last mentioned lemma, we may conclude that $H_{n,m}x_\kindex(C_{n,m},A) \to H_{n,m}x(C_{n,m},A)$ as $k \to \infty$.

To finish the argument, we employ Lemma~\ref{l:blockconvergence} to see that
\[
\sum_{\gamma \in C_{n,m}} \xi_{\delta_n}(\norm{Q_\gamma x_\kindex}_\gamma) \to \sum_{\gamma \in C_{n,m}} \xi_{\delta_n}(\norm{Q_\gamma x}_\gamma) \mbox{ as } \kindex \to \infty
\]
and we apply Lemma~\ref{l:gsystem} (A6) on the function $g_{n,m,\cardinality{C_{n,m}}}$. This leads to $\xi_{\delta_n}(\norm{P_A x_\kindex-x_\kindex}) \to \xi_{\delta_n}(\norm{P_A x-x})$ which means that $\norm{P_A x_\kindex - x_\kindex} \to \norm{P_A x - x}$ by our choice of $n$.
\end{proof}

In the end of all we are going to show that $J$ is a limit of \ksmooth{k} norms.
A self-evident choice for the approximating norms $J_i$ is as follows.
Let us define
\[
H^i_{n,m}x(A,B):=g_{n,m,l}\left(\sum_{\gamma \in A} \xi_{\delta_n}(\norm{Q_\gamma x}_{\gamma,i}),\xi_{\delta_n}( \norm{P_B x-x})\right),
\]
\[
J_{j,n,m,i}(x):=\xi_{\sigma_j}(Z_{\eta_{n,m}}(H^i_{n,m}x)),
\]
\[
\tilde{J}^2_i(x):=\norm{x}^2+\sum_{1\leq  j,n,m  \leq i} \frac{1}{2^{j+n+m}} J_{j,n,m,i}^2(x)
\]
and $J_i$ as the Minkowski functional of $\set{x \in X: \tilde{J}_i(x) \leq 1/2}$.
As a finite sum of \ksmooth{k} functions, $\tilde{J}_i$ is \ksmooth{k}. The Implicit Function Theorem implies the same about $J_i$. Moreover $2\norm{x} \leq J_i(x) \leq 2K''\norm{x}$ as in the proof of Lemma~\ref{l:JisEquivNorm}.
Let $\varepsilon>0$ be given. We will show that there is an index $i_0 \in \Natural$ such that $\abs{\tilde{J}^2_i(x)-\tilde{J}^2(x)}<\varepsilon$ whenever $\norm{x}<1$ and $i \geq i_0$. For this it is sufficient that $\left(\frac{2C}{i_0}\right)^2<\varepsilon/2$ and
\[\label{e:similarconsiderations}
\sum_{\max\set{j,n,m} \geq i_0} \frac{2}{2^{j+n+m}}<\varepsilon/2
\]
because then, for each $i \geq i_0$,
\[
\begin{split}
\abs{\tilde{J}^2_i(x)-\tilde{J}^2(x)}&\leq \sum_{1\leq  j,n,m  \leq i} \frac{1}{2^{j+n+m}} (J_{j,n,m,i}^2(x)-J_{j,n,m}^2(x)) + \sum_{\max\set{j,n,m} \geq i_0} \frac{1}{2^{j+n+m}}J_{j,n,m}^2(x)\\
&< \sum_{1\leq  j,n,m  \leq i} \frac{1}{2^{j+n+m}} \left(\frac{2C i}{i^2}\right)^2+\varepsilon/2 < \varepsilon
\end{split}
\]
where in the second inequality we are using \eqref{e:inductiveHypo} and $(1-\frac{1}{i^2})\norm{x}_\gamma \leq \norm{x}_{\gamma,i}\leq \norm{x}_\gamma$ to estimate the first term and $J_{j,n,m}(x) \leq 2$ for $\norm{x} < 1$.
This proves that $\tilde{J}_i \to \tilde{J}$ uniformly on $\openball{(X,\norm{\cdot})}$.

Now let us observe that, since $\tilde{J}(0)=0$, we have the estimate
\begin{equation}\label{e:MinkEstimate}
\frac{1}{2}\abs{\lambda-1} \leq \abs{\frac{1}{2}-\tilde{J}(\lambda x)}
\end{equation}
for all $x \in X$ such that $\tilde{J}(x)=\frac{1}{2}$, or equivalently such that $J(x)=1$.

We assume that there is a sequence $(x_i) \subset \openball{(X,\norm{\cdot})}$ such that $J_i(x_i) - J(x_i) \nrightarrow 0$. Let $c_i>0$, resp. $d_i>0$, be such that $J(c_i x_i)=1$, resp. $J_i(d_i x_i)=1$. It follows that $\lambda_i:=\frac{d_i}{c_i} \nrightarrow 1$ so we may and do assume that there is some $\varepsilon>0$ such that $\abs{\lambda_i-1}>2\varepsilon$ for all $i \in \Natural$.
On the other hand, since $\norm{d_i x_i} \leq \frac{1}{2}$ and since $\tilde{J}_i \to \tilde{J}$ uniformly on $\openball{(X,\norm{\cdot})}$, we get that $\abs{\tilde{J}(\lambda_i c_i x_i)-\frac{1}{2}}=\abs{\tilde{J}(\lambda_i c_i x_i)-\tilde{J}_i(d_i x_i)}\leq \varepsilon$ for $i$ large enough. Thus, having in mind \eqref{e:MinkEstimate}, we obtain $\abs{\lambda_i-1}\leq 2\varepsilon$. As a result of this contradiction we see immediatelly that $J_i \to J$ uniformly on bounded sets.


\begin{thebibliography}{77}  

\bibitem{A} E. Asplund, \textit{Averaged norms}, Isr. J. Math. 5 (1967), 227--233.



\bibitem{DGZ} R. Deville, G. Godefroy and V. Zizler, \textit{Smoothness and
renormings in Banach spaces}, Pitman Monographs and Surveys {\bf 64}, Longman
Ed (1993).



\bibitem{DFH1} R. Deville, V. Fonf and P. H\' ajek, \textit{Analytic and $C^k$-smooth approximations of
norms in separable Banach spaces}, Studia Math. 120, (1996), 61--74.



\bibitem{DFH2} R. Deville, V. Fonf and P. H\' ajek, \textit{Analytic and polyhedral approximations of
convex bodies in separable polyhedral Banach spaces} Israel
J. Math. 105 (1998), 139--154.




\bibitem{FHHMPZ}
{M. Fabian, P. Habala, P. H\'ajek, V. Montesinos, J. Pelant and V. Zizler},
\textit{Functional analysis and infinite dimensional geometry},  CMS Books in
Mathematics (Springer-Verlag) (2001).

\bibitem{FZZ} M. Fabian, L. Zaj\' i\v cek and V. Zizler, \textit{On residuality of the set of rotund norms on a Banach space},
Math. Ann. 258 (1981), 349--351.


\bibitem{FHZ}
{M. Fabian, P. H\'ajek, V. Zizler}, \textit{A note on lattice renormings}, Comment. Math. Univ. Carolinae 38, 2 (1997), 263-272.

\bibitem{FMZ}
{M. Fabian, V. Montesinos, V. Zizler}, \textit{Smoothness in Banach spaces. Selected problems}, Rev. R. Acad. Cien. Seria A. Mat. 100 (1-2), 2006, 101-125.

\bibitem{FWZ}
{M. Fabian, J.H.M. Whitfield, V. Zizler}, \textit{Norms with locally Lipschitzian properties},
Isr. J. Math. 44 (1983), 262--276.


\bibitem{F1}
{R. Fry},
\textit{Approximation by functions with bounded derivative on Banach spaces},
Bull. Austral. Math. Soc. 69 (2004), 125--131.
%
\bibitem{F2}
{R. Fry},
\textit{Approximation by Lipschitz, $C^p$ smooth functions on weakly compactly generated Banach spaces},
J. Funct. Anal. 252 (2007), 34--41.
%
\bibitem{HJ1}
{P. H\'ajek and M. Johanis},
\textit{Uniformly G\^ateaux smooth approximations on $c_0(\Gamma)$},
J. Math. Anal. Appl. 350 (2009), 623-629.
%

\bibitem{HJ2}
{P. H\' ajek and M. Johanis}, \textit{Lipschitz smooth approximations}, preprint.



\bibitem{H1} P. H\' ajek, \textit{On convex functions in $c_0(\omega_1)$}, Collect. Math. 47, 2 (1996),
111--115.


\bibitem{H2} P. H\' ajek, \textit{Smooth functions on $c_0$}, Israel J. Math. 104 (1998), 17--27.



\bibitem{HMVZ}
{P. H\' ajek, V. Montesinos, J. Vanderwerff and V. Zizler},
\textit{Biorthogonal systems in Banach spaces}, CMS Books in Mathematics
(Springer-Verlag) (2007).

\bibitem{Hay}
{R. Haydon}, \textit{Smooth functions and partitions of unity on certain Banach spaces},
The Quarterly Journal of Mathematics 1996 47(188):455-468; doi:10.1093/qjmath/47.188.455


\bibitem{J}
{M. Johanis},
\textit{Approximation of Lipschitz Mappings},
Serdica Math. J. 29, No. 2 (2003), 141--148.
%
\bibitem{JTZ}
{K. John, H. Toru\'nczyk and V. Zizler},
\textit{Uniformly smooth partitions of unity on superreflexive Banach spaces},
Studia Math. 70 (1981), 129--137.

\bibitem{MOTV}
{A. Molto, J. Orihuela, S. Troyanski and M. Valdivia},
\textit{A Nonlinear Transfer Technique for Renorming}, Springer LNM 1951 (2009).

\bibitem{MPVZ}
{D. McLaughlin, R. Poliquin, J. Vanderwerff and V. Zizler}, \textit{Second-order G\textasciicircum{a}eaux differentiable bump functions and approximations in Banach spaces}, Can. J. Math. 45 (3), 1993, 612-625.

\bibitem{T}
{M. Talagrand}, \textit{Renormages de quelques ${\mathcal C}(K)$}, Israel J. Math. 54, No. 3 (1986), 327--334.

\bibitem{To}
{H. Torunczyk}, \textit{Smooth partitions of unity on some nonseparable Banach spaces}, Studia Math. 46 (1973), 43--51.

\bibitem{Tr}
{S. Troyanski}, \textit{On locally uniformly convex and differentiable norms in certain non-separable Banach spaces}, Studia Math., 37 (1971), 173-180.

\end{thebibliography}
\end{document}